\documentclass[a4paper, oneside, 11pt]{amsart}
\usepackage{amsmath,amsfonts,amssymb,amsthm}
\usepackage{mathrsfs}
\usepackage{enumerate}
\usepackage[all]{xy}
\usepackage{booktabs}
\usepackage{longtable}
\usepackage{minibox}
\usepackage{bookmark}

\newtheorem{thm}{Theorem}[section]
\newtheorem{prop}[thm]{Proposition}

\newtheorem{cor}[thm]{Corollary}
\newtheorem{fact}[thm]{Fact}

\theoremstyle{remark}
\newtheorem{rem}[thm]{Remark}

\newcommand{\N}{\mathbb{N}}

\newcommand{\R}{\mathbb{R}}
\newcommand{\C}{\mathbb{C}}
\renewcommand{\H}{\mathbb{H}}

\newcommand{\e}{\mathfrak{e}}
\newcommand{\f}{\mathfrak{f}}
\newcommand{\g}{\mathfrak{g}}
\newcommand{\h}{\mathfrak{h}}
\renewcommand{\j}{\mathfrak{j}}
\renewcommand{\k}{\mathfrak{k}}
\renewcommand{\l}{\mathfrak{l}}

\newcommand{\q}{\mathfrak{q}}

\newcommand{\gl}{\mathfrak{gl}}
\renewcommand{\sl}{\mathfrak{sl}}
\newcommand{\su}{\mathfrak{su}}
\newcommand{\so}{\mathfrak{so}}
\renewcommand{\sp}{\mathfrak{sp}}
\newcommand{\To}{\sqrt{-1}\R}

\DeclareMathOperator{\SL}{SL}

\newcommand{\im}{\operatorname{im}}
\newcommand{\sgn}{\operatorname{sgn}}
\newcommand{\rank}{\operatorname{rank}}
\newcommand{\rest}{\operatorname{rest}}
\newcommand{\Stab}{\operatorname{Stab}}

\newcommand{\w}{\wedge}
\newcommand{\ox}{\otimes}
\newcommand{\bs}{\backslash}
\newcommand{\bl}{\bullet}

\newcommand{\injto}{\hookrightarrow}

\newcommand{\simto}{\xrightarrow{\sim}}
\renewcommand{\geq}{\geqslant}
\renewcommand{\leq}{\leqslant}

\begin{document}

\title[Semisimple symmetric spaces]{Semisimple symmetric spaces
that do not model any compact manifold}
\author{Yosuke Morita}
\address{Department of Mathematics, Graduate School of Science,
Kyoto University, Kitashirakawa Oiwake-cho, Sakyo-ku,
Kyoto 606-8502, Japan}
\email{yosuke.m@math.kyoto-u.ac.jp}

\begin{abstract}
In a previous paper,
we obtained a cohomological obstruction to the existence of
compact manifolds locally modelled on a homogeneous space.
In this paper, we give a classification of the semisimple
symmetric spaces to which this obstruction is applicable.
\end{abstract}

\maketitle

\section{Introduction}

A manifold is said to be locally modelled on a homogeneous space
$G/H$ if it is covered by open sets that are diffeomorphic to
open sets of $G/H$ and
the transition functions are left translations by elements of $G$.
A basic example of a manifold locally modelled on $G/H$
is a Clifford--Klein form, i.e.\ a quotient space $\Gamma \bs G/H$,
where $\Gamma$ is a discrete subgroup of $G$
acting properly and freely on $G/H$.
We always assume that
the transition functions satisfy the cocycle condition
(see \cite[\S 2]{Mor17PRIMS}).
This assumption is automatically satisfied if $G/H$
is connected and $G$ acts effectively on $G/H$.
Also, every Clifford--Klein form satisfies this assumption.

Since T. Kobayashi's initial paper \cite{Kob89Ann},
the study of topology of Clifford--Klein forms
has attracted considerable attention, and in particular,
various obstructions to the existence of \emph{compact}
Clifford--Klein forms of $G/H$ (or, more generally,
of \emph{compact} manifold locally modelled on $G/H$)
has been found
(see surveys \cite{Kob97}, \cite{Kob05}, \cite{Kob-Yos05},
\cite{Lab96} and references therein).

In a previous paper \cite{Mor15JDG},
we obtained the following obstruction:

\begin{fact}\label{fact:previous}
Let $G/H$ be a homogeneous space of reductive type and
$K_H$ a maximal compact subgroup of $H$.
Let $\g$, $\h$ and $\k_H$ denote the Lie algebras of
$G$, $H$ and $K_H$, respectively. If the homomorphism
$i: H^\bl(\g, \h; \R) \to H^\bl(\g, \k_H; \R)$
induced from the inclusion map
$(\Lambda (\g/\h)^\ast)^\h \injto (\Lambda (\g/\k_H)^\ast)^{\k_H}$
is not injective, then there does not exist
a compact manifold locally modelled on the homogeneous space $G/H$
(and, in particular,
there does not exist a compact Clifford--Klein form of $G/H$).
\end{fact}

\begin{rem}
\begin{enumerate}[(1)]
\item In \cite{Mor15JDG}, Fact~\ref{fact:previous}
was proved only for the case of Clifford--Klein forms,
but the same proof applies to the case of
manifolds locally modelled on $G/H$ (see \cite{Mor17PRIMS}).
Note that we are not sure if this is an essential generalization,
as we do not know any example of a compact manifold
locally modelled on a homogeneous space of reductive type
that is not a Clifford--Klein form.
\item For the case of Clifford--Klein forms,
Fact~\ref{fact:previous} was generalized by Tholozan \cite{Tho15+}
and the author \cite{Mor17Selecta}.
\end{enumerate}
\end{rem}

The purpose of this paper is to classify
the semisimple symmetric spaces $G/H$
to which Fact~\ref{fact:previous} is applicable.
Our main result is the following:

\begin{thm}\label{thm:classify}
For a semisimple symmetric pair $(\g, \h)$,
the following two conditions are equivalent:
\begin{enumerate}
\item[\textup{(A)}] The homomorphism
$i: H^\bl(\g, \h; \R) \to H^\bl(\g, \k_H; \R)$
induced from the inclusion map
$(\Lambda (\g/\h)^\ast)^\h \injto (\Lambda (\g/\k_H)^\ast)^{\k_H}$
is injective.
\item[\textup{(B)}] The pair $(\g, \h)$ is isomorphic
(up to possibly outer automorphisms)
to a direct sum of the following irreducible symmetric pairs
\textup{(B-1)}--\textup{(B-5)}.
\begin{enumerate}
\item[\textup{(B-1)}] $(\l, \l)$ ($l$: simple Lie algebra).
\item[\textup{(B-2)}] $(\l \oplus \l, \Delta \l)$
($\l$: simple Lie algebra).
\item[\textup{(B-3)}] $(\l_\C, \l)$
($\l_\C$: complex simple Lie algebra, $\l$: real form of $\l_\C$).
\item[\textup{(B-4)}] An irreducible symmetric pair $(\g', \h')$
with $\rank \h' = \rank \k_{H'}$,
where $\k_{H'}$ is a maximal compact subalgebra of $\h'$.
\item[\textup{(B-5)}] \begin{itemize}
\item $(\sl(2n+1, \C), \so(2n+1, \C))$ $(n \geq 1)$,
\item $(\sl(2n, \C), \sp(n, \C))$ $(n \geq 2)$,
\item $(\so(2n, \C), \so(2n-1, \C))$ $(n \geq 3)$,
\item $(\e_{6,\C}, \f_{4,\C})$.
\end{itemize}
\end{enumerate}
\end{enumerate}
\end{thm}

\begin{rem}
By Fact~\ref{fact:previous}, there exists a compact manifold
locally modelled on a semisimple symmetric space $G/H$
only if the corresponding semisimple symmetric pair $(\g, \h)$
satisfies \textup{(B)}.
So far, we cannot exclude the possibility that the converse holds,
but it seems unlikely to be true.
At least, it is false for Clifford--Klein forms:
many of the irreducible symmetric spaces in (B-3)--(B-5)
do not to admit compact Clifford--Klein forms
(\cite{Kob89Ann}, \cite{Kob92Duke}, \cite{Ben96}, \cite{Oku13},
\cite{Tho15+}, \cite{Mor17Selecta}).
Note that all of the irreducible symmetric spaces in (B-1)--(B-2)
admit compact Clifford--Klein forms.
\end{rem}

In Table~\ref{table:classify},
we list all the irreducible symmetric pairs $(\g, \h)$
that do not satisfy the condition \textup{(B)} among Berger's
classification of the irreducible symmetric pairs~\cite{Ber57}.
By Fact~\ref{fact:previous},
there does not exist a compact manifold locally modelled on
an irreducible symmetric space $G/H$ if the corresponding pair
$(\g, \h)$ is listed in Table~\ref{table:classify}.

\medskip

\begin{center}
\begin{longtable}{llll} \toprule
& \multicolumn{1}{c}{$\g$} & \multicolumn{1}{c}{$\h$} & \multicolumn{1}{c}{Conditions} \\ \midrule
$\star$      & $\sl(2n, \C)$ & $\so(2n, \C)$
                    & $n \geq 1$ \\ \hline
             & $\sl(p+q, \C)$ & $\sl(p, \C) \oplus \sl(q, \C) \oplus \C$
                    & $p,q \geq 1$ \\ \hline
$\star\star$ & $\sl(p+q, \R)$ & $\so(p, q)$
                    & $p,q \geq 1$, $p,q$: odd \\ \hline
$\star$      & $\su(p, q)$ & $\so(p, q)$
                    & $p,q \geq 1$, $p,q$: odd \\ \hline
             & $\su(n, n)$ & $\sl(n, \C) \oplus \R$
                    & $n \geq 1$ \\ \hline
$\star\star$ & $\sl(2n, \R)$ & $\sl(n, \C) \oplus \To$
                    & $n \geq 2$ \\ \hline
$\star$      & $\sl(n, \H)$ & $\sl(n, \C) \oplus \To$
                    & $n \geq 2$ \\ \hline
             & $\sl(p+q, \R)$ & $\sl(p, \R) \oplus \sl(q, \R) \oplus \R$
                    & $p,q \geq 1$ \\ \hline
             & $\sl(p+q, \H)$ & $\sl(p, \H) \oplus \sl(q, \H) \oplus \R$
                    & $p,q \geq 1$ \\ \hline
$\star$      & $\so(p+q, \C)$ & $\so(p, \C) \oplus \so(q, \C)$
                    & $p,q \geq 2$, $(p,q) \neq (2,2)$ \\ \hline
$\star$      & $\so(2n+1, \C)$ & $\so(2n, \C)$
                    & $n \geq 1$ \\ \hline
             & $\so(2n, \C)$ & $\sl(n, \C) \oplus \C$
                    & $n \geq 3$ \\ \hline
$\star\star$ & $\so(n, n)$ & $\so(n, \C)$
                    & $n \geq 3$ \\ \hline
$\star$      & $\so^\ast(2n)$ & $\so(n, \C)$
                    & $n \geq 3$ \\ \hline
                & & & $p,q \geq 1$, $p,q$: odd, \\
$\star\star$ & $\so(p+r,q+s)$ & $\so(p,q) \oplus \so(r,s)$
                    & $r,s \geq 0$, $(r,s) \neq (0,0)$, \\
                & & & $(p,q,r,s) \neq (1,1,1,1)$ \\ \hline
             & $\so(n, n)$ & $\sl(n, \R) \oplus \R$
                    & $n \geq 3$ \\ \hline
             & $\so^\ast(4n)$ & $\sl(n, \H) \oplus \R$
                    & $n \geq 2$ \\ \hline
             & $\sp(n, \C)$ & $\sl(n, \C) \oplus \C$
                    & $n \geq 1$ \\ \hline
$\star$      & $\sp(p+q, \C)$ & $\sp(p, \C) \oplus \sp(q, \C)$
                    & $p, q \geq 1$ \\ \hline
$\star$      & $\sp(2n, \R)$ & $\sp(n, \C)$
                    & $n \geq 1$ \\ \hline
$\star$      & $\sp(n, n)$ & $\sp(n, \C)$
                    & $n \geq 1$ \\ \hline
             & $\sp(n, \R)$ & $\sl(n, \R) \oplus \R$
                    & $n \geq 1$ \\ \hline
             & $\sp(n, n)$ & $\sl(n, \H) \oplus \R$
                    & $n \geq 1$ \\ \hline
$\star$      & $\e_{6,\C}$ & $\sp(4, \C)$ & --- \\ \hline
$\star$      & $\e_{6,\C}$ & $\sl(6, \C) \oplus \sl(2, \C)$ & --- \\ \hline
             & $\e_{6,\C}$ & $\so(10, \C) \oplus \C$ & --- \\ \hline
$\star$      & $\e_{6(6)}$ & $\sl(6, \R) \oplus \sl(2, \R)$ & --- \\ \hline
$\star\star$ & $\e_{6(6)}$ & $\sl(3, \H) \oplus \su(2)$ & --- \\ \hline
$\star$      & $\e_{6(-26)}$ & $\sl(3, \H) \oplus \su(2)$ & --- \\ \hline
             & $\e_{6(6)}$ & $\so(5,5) \oplus \R$ & --- \\ \hline
             & $\e_{6(-26)}$ & $\so(9,1) \oplus \R$ & --- \\ \hline
$\star$      & $\e_{7,\C}$ & $\sl(8, \C)$ & --- \\ \hline
$\star$      & $\e_{7,\C}$ & $\so(12, \C) \oplus \sl(2, \C)$ & --- \\ \hline
             & $\e_{7,\C}$ & $\e_{6,\C} \oplus \C$ & --- \\ \hline
             & $\e_{7(7)}$ & $\sl(8, \R)$ & --- \\ \hline
             & $\e_{7(7)}$ & $\sl(4, \H)$ & --- \\ \hline
             & $\e_{7(-25)}$ & $\sl(4, \H)$ & --- \\ \hline
             & $\e_{7(7)}$ & $\e_{6(6)} \oplus \R$ & --- \\ \hline
             & $\e_{7(-25)}$ & $\e_{6(-26)} \oplus \R$ & --- \\ \hline
$\star$      & $\e_{8,\C}$ & $\so(16, \C)$ & --- \\ \hline
$\star$      & $\e_{8,\C}$ & $\e_{7,\C} \oplus \sl(2, \C)$ & --- \\ \hline
$\star$      & $\f_{4,\C}$ & $\sp(3, \C) \oplus \sl(2, \C)$ & --- \\ \hline
$\star$      & $\f_{4,\C}$ & $\so(9, \C)$ & --- \\ \hline
$\star$      & $\g_{2,\C}$ & $\sl(2, \C) \oplus \sl(2, \C)$ & --- \\
\bottomrule
\caption{$(\g, \h)$ not satisfying \textup{(B)}}
\label{table:classify}
\end{longtable}
\end{center}

\vspace{-1.2cm}

In Table~\ref{table:classify},
the signs $\star\star$ and $\star$ signify
\begin{itemize}
\item[$\star\star$:]
The nonexistence of compact Clifford--Klein forms of $G/H$
seems to be not known before \cite{Mor15JDG}.
\item[$\star$:]
The nonexistence of compact Clifford--Klein forms of $G/H$
had been known before \cite{Mor15JDG},
but not for the locally modelled case.
\end{itemize}
Note that we saw in \cite[Cor.\ 1.4]{Mor15JDG}
the nonexistence of compact Clifford--Klein forms of $\star\star$
except for the case $(\e_{6(6)}, \sl(3, \H) \oplus \su(2))$.

\begin{rem}
Our proof of Theorem~\ref{thm:classify} relies on Berger's
classification of the irreducible symmetric pairs~\cite{Ber57}.
Thus, we do not know why the conditions (B-1)--(B-5)
looks relatively simple. In particular,
we do not know why the pairs listed in the (B-5) are all complex.
\end{rem}

\begin{rem}[cf.\ {\cite[\S 0.1.5]{Kas09},
\cite[\S 4]{Kob01unlimited}, \cite[\S 4.3]{Kob05}}]
We mention some previous works on
the existence problem of compact Clifford--Klein forms.
To the best of the author's knowledge,
there are basically four methods to study this problem
(this paper is based on the method (ii)):
\begin{enumerate}[(i)]
  \item A criterion for properness in terms of the Cartan projection
  (e.g.\ \cite{Kob89Ann}, \cite{Kob92Duke}, \cite{Ben96}, \cite{Oku13}).
  \item Comparison of relative Lie algebra cohomology and de Rham cohomology
  (e.g.\ \cite{Kob-Ono90}, \cite{Ben-Lab92}, \cite{Mor15JDG},
  \cite{Tho15+}, \cite{Mor17Selecta}).
  \item Zimmer's cocycle superrigidity
  (e.g.\ \cite{Zim94}, \cite{Cor-Zim94}, \cite{LMZ95}, \cite{Lab-Zim95}).
  \item Estimate of the decay of matrix coefficients
  (e.g.\ \cite{Mar97}, \cite{Oh98}, \cite{Sha00}).
\end{enumerate}
Each of them has its own advantage and applies to different examples.
For example, in the case of $\SL(n,\R)/\SL(m,\R) \ (n>m \geq 2)$,
where $\SL(m,\R)$ is put in the upper-left corner of $\SL(n,\R)$
unless otherwise stated, each method gives the following results:
\begin{itemize}
  \item Kobayashi~\cite{Kob92Duke} (method (i)):
  If $(n-1)/3 \geq \lfloor (m+1)/2 \rfloor$ and $m \geq 2$,
  then $\SL(n,\R)/\SL(m,\R)$ does not have a compact Clifford--Klein form.
  \item Zimmer~\cite{Zim94}, Labourie--Mozes--Zimmer~\cite{LMZ95},
  Labourie--Zimmer~\cite{Lab-Zim95} (method (iii)):
  If $n-3 \geq m \geq 2$, then $\SL(n,\R)/\SL(m,\R)$ does not have a compact Clifford--Klein form.
  If, in addition, $n \geq 2m$,
  then there does not exist a compact manifold locally modelled on $G/H$.
  \item Benoist~\cite{Ben96} (method (i)):
  If $m \geq 2$ is even, then every discrete subgroup $\Gamma$ of $\SL(m+1,\R)$
  acting properly and freely on $\SL(m+1,\R)/\SL(m,\R)$ is virtually abelian, and in particular, $\SL(m+1,\R)/\SL(m,\R)$ does not have a compact Clifford--Klein form.
  \item Margulis~\cite{Mar97} (method (iv)):
  Let $\alpha_n$ be the irreducible representation of $\SL(2, \R)$ into $\SL(n, \R)$.
  If $n \geq 4$, then $\SL(n, \R)/\alpha_n(\SL(2, \R))$ does not admit a compact Clifford--Klein form.
  \item Shalom~\cite{Sha00} (method (iv)):
  If $n \geq 4$, then $\SL(n,\R)/\SL(2,\R)$ does not have a compact Clifford--Klein form.
  \item Tholozan~\cite{Tho15+}, Morita~\cite{Mor17Selecta} (method (ii)):
  If $m \geq 2$ is even and $n>m$, then $\SL(n,\R)/\SL(m,\R)$ does not have a compact Clifford--Klein form.
\end{itemize}
We remark that $\SL(n,\R)/\SL(m,\R)$ and
$\SL(n,\R)/\alpha_n(\SL(2,\R))$ are not semisimple symmetric spaces.
Among these methods, (iii) and (iv)
are not applicable to the case of semisimple symmetric spaces,
whereas they sometimes give sharper results in the nonsymmetric case.
\end{rem}

Main results of this paper, along with an outline of proofs,
were announced in \cite{Mor15Proc}.
We shall give somewhat simpler proofs based on a necessary and
sufficient condition for the injectivity of the homomorphism
$i : H^\bl(\g, \h; \R) \to H^\bl(\g, \k_H; \R)$
obtained in \cite{Mor17+} (Fact~\ref{fact:CO}).
The author apologizes for the long delay in writing this paper.

\begin{rem}
We take this opportunity to correct some errors in
\cite{Mor15Proc}:
\begin{itemize}
\item In \cite[Th.\ 2.1]{Mor15Proc}, the pair
$(\sl(2n,\C), \sp(n,\C))$ ($n \geq 1$) should be
$(\sl(2n,\C), \sp(n,\C))$ ($n \geq 2$).
\item The pairs
$(\sp(2n, \R), \sp(n, \C))$ ($n \geq 1$) and
$(\sp(n, n), \sp(n, \C))$ ($n \geq 1$)
should be labelled as $\star$.
\item Since $\so(4,\C)$, $\so(2,2)$ and $\so^\ast(4)$ are
not simple Lie algebras,
the pairs
$(\g, \h)$ with $\g = \so(4,\C)$, $\so(2,2)$ or $\so^\ast(4)$
should not be listed in the table.
\end{itemize}
\end{rem}

\section{Preliminaries}

\subsection{Reductive pairs and semisimple symmetric pairs}

We say that $(\g, \h)$ is a reductive pair if $\g$
is a (real) reductive Lie algebra with Cartan involution $\theta$
and $\h$ is a subalgebra of $\g$ such that $\theta(\h) = \h$.
We then put $\k = \g^\theta$ and $\k_H = \h^\theta$.
Similarly, we say that a homogeneous space $G/H$
is of reductive type if $G$ is a linear reductive Lie group with
Cartan involution $\theta$ and $H$ is a closed subgroup of $G$
with finitely many connected components such that $\theta(H) = H$.
We then put $K = G^\theta$ and $K_H = H^\theta$.
Note that $K$ and $K_H$ are maximal compact subgroups of
$G$ and $H$, respectively.

If $\g$ is a semisimple Lie algebra and
$\h = \g^\sigma$ for some involution $\sigma$ of $\g$,
we call $(\g, \h)$ a semisimple symmetric pair.
In this case, there exists a Cartan involution $\theta$ of $\g$
such that $\theta \sigma = \sigma \theta$
(\cite[Lem.\ 10.2]{Ber57}),
and any other Cartan involution that commutes with $\sigma$
is of the form $\exp(X) \theta \exp(-X)$, where $X \in \h$
(\cite[Lem.\ 4]{Mat79}).
Therefore, $(\g, \h)$ can be seen as a reductive pair
in a natural way.
Similarly, if $G$ is a connected linear semisimple Lie group and
$H$ is an open subgroup of $G^\sigma$ for some involution $\sigma$
of $G$, we call $G/H$ a semisimple symmetric space, which
has a natural structure of homogeneous space of reductive type.
We say that a semisimple symmetric pair $(\g, \h)$
is an irreducible symmetric pair if $\g$ is simple or $(\g, \h)$
is isomorphic (up to possibly outer automorphisms) to
$(\l \oplus \l, \Delta \l)$ for some simple Lie algebra $\l$.
A semisimple symmetric space is called
an irreducible symmetric space
if the corresponding semisimple symmetric pair is irreducible.
Every semisimple symmetric pair is uniquely decomposed into
irreducible ones.
The complete classification of the irreducible symmetric pairs
(up to possibly outer automorphisms)
is accomplished by Berger~\cite{Ber57}.

\subsection{The graded algebra $(S \g^\ast)^\g$ and the graded vector space $P_{\g^\ast}$}

Let us recall some basic results on
the algebra $(S \g^\ast)^\g$ of $\g$-invariant polynomials on $\g$.

\begin{fact}[Chevalley restriction theorem, see e.g.\ {\cite[Ch.~VIII, \S8, no.~3, Th.~1]{BouLie7-9}}]\label{fact:CRT}
Let $\g$ be a (real or complex) reductive Lie algebra.
Let $\j$ be a Cartan subalgebra of $\g$ and
$W$ the associated Weyl group.
Then, the restriction map
\[ \rest : (S\g^\ast)^\g \to (S \j^\ast)^W \]
is an isomorphism.
\end{fact}
\begin{fact}[Chevalley--Shepherd--Todd, see e.g.\
{\cite[Th.~3.5 and Prop.~3.7]{Hum90}}]\label{fact:CST}
Let $\g$ be a (real or complex) reductive Lie algebra
of rank $r$.
Let $\j$ be a Cartan subalgebra of $\g$ and
$W$ the associated Weyl group.
Then, $(S \j^\ast)^W$ (or equivalently, $(S \g^\ast)^\g$)
is generated by $r$
algebraically independent homogeneous elements $(P_1, \dots, P_r)$
of positive degree.
The degrees $(\deg P_1, \dots, \deg P_r)$
do not depend on the choice of $(P_1, \dots, P_r)$.
\end{fact}

For a complex simple Lie algebra $\g$,
the degrees of algebraically independent generators of
$(S \g^\ast)^\g \ (\simeq (S \j^\ast)^W)$ are as follows
(see e.g.\ \cite[p.~59]{Hum90}):

\begin{center}
\begin{longtable}{ll} \toprule
\multicolumn{1}{c}{$\g$} & \multicolumn{1}{c}{Degrees} \\ \midrule
$\sl(n, \C)$ & $2, 3, \dots, n$ \\ \hline
$\so(2n+1, \C)$ & $2, 4, \dots, 2n$ \\ \hline
$\sp(n, \C)$ & $2, 4, \dots, 2n$ \\ \hline
$\so(2n, \C)$ & $2, 4, \dots, 2n-2, n$ \\ \hline
$\e_{6,\C}$ & $2,5,6,8,9,12$ \\ \hline
$\e_{7,\C}$ & $2,6,8,10,12,14,18$ \\ \hline
$\e_{8,\C}$ & $2,8,12,14,18,20,24,30$ \\ \hline
$\f_{4,\C}$ & $2,6,8,12$ \\ \hline
$\g_{2,\C}$ & $2,6$ \\
\bottomrule
\caption{Degrees of generators of $(S \g^\ast)^\g$}
\label{table:degree-sym}
\end{longtable}
\end{center}

\vspace{-1.2cm}

If $\g$ is abelian,
the graded algebra $(S \g^\ast)^\g = S \g^\ast$
is generated by the elemnts of degree $1$.

For the classical cases, the structure of graded algebra
$(S \g^\ast)^\g$ is explicitly described as follows:
\begin{fact}[see e.g.\ {\cite[Ch.~VIII, \S 13]{BouLie7-9}}]\label{fact:f_k}
Let $f_k \in (S^k (\gl(n, \C)^\ast))^{\gl(n, \C)}$
$(k=1, 2, \dots, n)$ denote invariant polynomials defined by
\[
\det(\lambda I_n - X) =
\lambda^n + f_1(X) \lambda^{n-1} + f_2(X) \lambda^{n-2}
+ \dots + f_n(X) \quad (X \in \gl(n, \C)).
\]
We use the same notation $f_k$ for the restriction of $f_k$ to
$\sl(n, \C)$, $\so(n, \C)$ or $\sp(m, \C)$ (if $n=2m$). Then,
\begin{itemize}
\item The graded algebra
$(S (\sl(n, \C)^\ast))^{\sl(n, \C)}$ is the polynomial algebra of
$(n-1)$ variables $f_2, f_3, \dots, f_n$.
We have $f_1 = 0$.
\item If $n = 2m+1$,
the graded algebra
$(S (\so(n, \C)^\ast))^{\so(n, \C)}$ is the polynomial algebra of
$m$ variables $f_2, f_4, \dots, f_{2m}$.
We have $f_1 = f_3 = \dots = f_{2m+1} = 0$.
\item If $n = 2m$,
the graded algebra
$(S (\sp(n, \C)^\ast))^{\sp(n, \C)}$ is the polynomial algebra of
$m$ variables $f_2, f_4, \dots, f_{2m}$.
We have $f_1 = f_3 = \dots = f_{2m-1} = 0$.
\item If $n = 2m$,
the graded algebra
$(S (\so(n, \C)^\ast))^{\so(n, \C)}$ is the polynomial algebra of
$m$ variables $f_2, f_4, \dots, f_{2m-2}, \tilde{f}$, where
$\tilde{f} \in (S^m (\so(n, \C)^\ast))^{\so(n, \C)}$
is the Pfaffian of $n \times n$ skew-symmeric matrices.
We have $f_1 = f_3 = \dots = f_{2m-1} = 0$ and
$f_{2m} = \tilde{f}^2$.
\end{itemize}
\end{fact}

Let us then recall the definition of
the space $P_{\g^\ast}$ of primitive elements in
$(\Lambda \g^\ast)^\g$, and its relation with the algebra
$(S \g^\ast)^\g$ (see \cite{GHV76} and \cite{Oni94} for details).

For a (real or complex) reductive Lie algebra $\g$,
we denote by $P_{\g^\ast}$
the space of primitive elements in $(\Lambda \g^\ast)^\g$, i.e.\
\[
P_{\g^\ast} = \{ \alpha \in (\Lambda^+ \g^\ast)^\g :
\text{$\alpha(x \w y) = 0$ for all
$x, y \in (\Lambda^+ \g)^\g$} \},
\]
where $\Lambda^+$ denotes the positive degree part of
the exterior algebra.
Then $P_{\g^\ast}$ is oddly graded:
$P_{\g^\ast} = \bigoplus_{k \geq 1} P_{\g^\ast}^{2k-1}$
(\cite[Ch.~5, \S 5]{GHV76}).

We define the Cartan map
$\rho_\g : (S^k \g^\ast)^\g \to (\Lambda^{2k-1} \g)^\g$
($k \geq 1$) by the following formula:
for $X_1, \dots, X_{2k-1} \in \g$,
\begin{align*}
&\rho_\g(P)(X_1, \dots, X_{2k-1}) \\
&\qquad = \sum_{\sigma \in \mathfrak{S}_{2k-1}}
\sgn(\sigma) P(X_{\sigma(1)},
[X_{\sigma(2)}, X_{\sigma(3)}], \dots,
[X_{\sigma(2k-2)}, X_{\sigma(2k-1)}]).
\end{align*}
\begin{rem}
Up to scalar multiple, the above definition of $\rho_\g$
coincides with the definition in \cite[Ch.~VI, \S2]{GHV76},
which uses the Weil algebra (see \cite[Ch.~VI, Prop.~IV]{GHV76})
\end{rem}
If $\g$ is reductive, $\rho_\g$ induces a linear isomorphism
\[
\overline{\rho_\g} :
(S^+ \g^\ast)^\g / ((S^+ \g^\ast)^\g \cdot (S^+ \g^\ast)^\g)
\simto P_{\g^\ast}
\]
(\cite[Ch.~VI, Th.~2]{GHV76}).
Therefore, one can derive the degrees of a basis of $P_{\g^\ast}$,
called the exponents of $\g$,
from Table~\ref{table:degree-sym}:

\begin{center}
\begin{longtable}{ll} \toprule
\multicolumn{1}{c}{$\g$} & \multicolumn{1}{c}{Degrees} \\ \midrule
$\sl(n, \C)$ & $3, 5, 7, \dots, 2n-1$ \\ \hline
$\so(2n+1, \C)$ & $3, 7, 11, \dots, 4n-1$ \\ \hline
$\sp(n, \C)$ & $3, 7, 11, \dots, 4n-1$ \\ \hline
$\so(2n, \C)$ & $3, 7, 11, \dots, 4n-5, 2n-1$ \\ \hline
$\e_{6,\C}$ & $3,9,11,15,17,23$ \\ \hline
$\e_{7,\C}$ & $3,11,15,19,23,27,35$ \\ \hline
$\e_{8,\C}$ & $3,15,23,27,35,39,47,59$\\ \hline
$\f_{4,\C}$ & $3,11,15,23$ \\ \hline
$\g_{2,\C}$ & $3,11$ \\
\bottomrule
\caption{Degrees of a basis of $P_{\g^\ast}$}
\label{table:degree}
\end{longtable}
\end{center}

\vspace{-1.2cm}

If $\g$ is abelian, the graded vector space $P_{\g^\ast}$ is
concentrated in degree $1$.

If $(\g, \h)$ is a reductive pair, the restriction map
$\rest : (\Lambda \g^\ast)^\g \to (\Lambda \h^\ast)^\h$
induces a linear map
\[
\rest : P_{\g^\ast} \to P_{\h^\ast}.
\]
Similarly, the restriction map
$\rest : (S \g^\ast)^\g \to (S \h^\ast)^\h$ induces a linear map
\[
\overline{\rest} :
(S^+ \g^\ast)^\g / ( (S^+ \g^\ast)^\g \cdot (S^+ \g^\ast)^\g ) \to
(S^+ \h^\ast)^\h / ( (S^+ \h^\ast)^\h \cdot (S^+ \h^\ast)^\h ).
\]
The following diagram commutes:
\[
\xymatrix{
(S^+ \g^\ast)^\g / ( (S^+ \g^\ast)^\g \cdot (S^+ \g^\ast)^\g ) \ar[r]^{\qquad\qquad\quad\ \sim}_{\qquad\qquad\quad\ \overline{\rho_\g}} \ar[d]_{\overline{\rest}} &
P_{\g^\ast} \ar[d]^{\rest} \\
(S^+ \h^\ast)^\h / ( (S^+ \h^\ast)^\h \cdot (S^+ \h^\ast)^\h ) \ar[r]^{\qquad\qquad\quad\ \sim}_{\qquad\qquad\quad\ \overline{\rho_\h}} &
P_{\h^\ast}
}
\]
(\cite[Ch.~VI, Prop.~2]{GHV76}).

\subsection{The injectivity of $i : H^\bl(\g, \h; \R) \to H^\bl(\g, \k_H; \R)$ and the graded vector space $(P_{\g^\ast})^{-\theta}$}

The following fact, proved in a previous paper \cite{Mor17+},
plays a foundational role in our classification:

\begin{fact}\label{fact:CO}
Let $(\g, \h)$ be a reductive pair with Cartan involution $\theta$.
Put $\k_H = \h^\theta$.
Then, the homomorphism
$i : H^\bl(\g, \h; \R) \to H^\bl(\g, \k_H; \R)$
is injective if and only if the linear map
$\rest: (P_{\g^\ast})^{-\theta} \to (P_{\h^\ast})^{-\theta}$
is surjective, where $(\, \cdot \,)^{-\theta}$ denotes
the $(-1)$-eigenspace for $\theta$.
\end{fact}
\begin{rem}
The isomorphism
\[
\overline{\rho_\g} :
(S^+ \g^\ast)^\g / ( (S^+ \g^\ast)^\g \cdot (S^+ \g^\ast)^\g )
\simto P_{\g^\ast}
\]
induced by the Cartan map commutes with the Cartan involution
$\theta$.
Therefore, Fact~\ref{fact:CO} is rephrased as follows:
\emph{the homomorphism
$i : H^\bl(\g, \h; \R) \to H^\bl(\g, \k_H; \R)$
is injective if and only if the linear map
\[
\overline{\rest} :
( (S^+ \g^\ast)^\g / ( (S^+ \g^\ast)^\g \cdot (S^+ \g^\ast)^\g ) )^{-\theta}
\to ( (S^+ \h^\ast)^\h / ( (S^+ \h^\ast)^\h \cdot (S^+ \h^\ast)^\h ) )^{-\theta}
\]
is surjective.}
See \cite{Mor17+} for some other conditions equivalent to
the injectivity of $i: H^\bl(\g, \h; \R) \to H^\bl(\g, \k_H; \R)$.
\end{rem}

\subsection{Some properties of $(P_{\g^\ast})^{-\theta}$}

We mention some results on $(P_{\g^\ast})^{-\theta}$
that are useful when we verify whether
$\rest: (P_{\g^\ast})^{-\theta} \to (P_{\h^\ast})^{-\theta}$
is surjective or not.

\begin{fact}[cf.\ {\cite[Ch.~X, Prop.~VII]{GHV76}}]\label{fact:GHV-X7-1}
Let $\g$ be a reductive Lie algebra with Cartan involution
$\theta$. Put $\k = \g^\theta$.
Let $\rho_\g : (S^+ \g^\ast)^\g \to (\Lambda^+ \g)^\g$
be the Cartan map for $\g$. Then,
\[
(P_{\g^\ast})^{-\theta} =
\{ \rho_\g(P) : P \in (S^+ \g^\ast)^\g,\ P|_{\k} = 0 \}.
\]
\end{fact}
\begin{fact}[{\cite[Ch.\ X, Cor.\ to Prop.\ VI]{GHV76}}]\label{fact:GHV-X7-2}
Let $\g$ be a reductive Lie algebra with Cartan involution
$\theta$. Put $\k = \g^\theta$. Then,
\[ \dim (P_{\g^\ast})^{-\theta} = \rank \g - \rank \k. \]
\end{fact}

\begin{prop}\label{prop:degree}
For a simple Lie algebra $\g$ with Cartan involution $\theta$,
the degrees of a basis of $(P_{\g^\ast})^{-\theta}$ are as follows:
\end{prop}
\begin{center}
\begin{longtable}{ll} \toprule
\multicolumn{1}{c}{$\g$} & \multicolumn{1}{c}{Degrees} \\ \midrule
$\sl(n, \R)$ & $5, 9, 13, \dots, 4 \lfloor \frac{n+1}{2} \rfloor -3$ \\ \hline
$\sl(n, \H)$ & $5, 9, 13, \dots, 4n-3$ \\ \hline
$\so(p,q)$ ($p,q$: odd) & $p+q-1$ \\ \hline
$\e_{6(6)}$ & $9,17$ \\ \hline
$\e_{6(-26)}$ & $9,17$ \\ \hline
$\sl(n, \C)$ & $3, 5, 7, \dots, 2n-1$ \\ \hline
$\so(2n+1, \C)$ & $3, 7, 11, \dots, 4n-1$ \\ \hline
$\sp(n, \C)$ & $3, 7, 11, \dots, 4n-1$ \\ \hline
$\so(2n, \C)$ & $3, 7, 11, \dots, 4n-5, 2n-1$ \\ \hline
$\e_{6,\C}$ & $3,9,11,15,17,23$ \\ \hline
$\e_{7,\C}$ & $3,11,15,19,23,27,35$ \\ \hline
$\e_{8,\C}$ & $3,15,23,27,35,39,47,59$\\ \hline
$\f_{4,\C}$ & $3,11,15,23$ \\ \hline
$\g_{2,\C}$ & $3,11$ \\ \hline
Otherwise & --- \\
\bottomrule
\caption{Degrees of a basis of $(P_{\g^\ast})^{-\theta}$}
\label{table:degree-anti}
\end{longtable}
\end{center}

\vspace{-1.2cm}

\begin{proof}
Although this proposition seems to be
already known in the early 1960s (cf.\ \cite{Tak62}),
we give its proof for the reader's convenience.

If $\rank \g = \rank \k$,
we have $(P_{\g^\ast})^{-\theta} = \{ 0 \}$ by
Fact~\ref{fact:GHV-X7-2}. We thus assume $\rank \g \neq \rank \k$.

Let $\g_\C$ denote the complexification of $\g$.
We use the same symbol $\theta$
for the complex linear extension of $\theta$ to $\g_\C$.
It suffices to compute the degrees of a basis of the
complexification
$(P_{\g^\ast})^{-\theta} \ox \C = (P_{\g_\C^\ast})^{-\theta}$.

Suppose that $\g$ is a complex simple Lie algebra.
We write $\g_\R$ for $\g$ regarded a real simple Lie algebra.
We have a natural isomorphism $(\g_\R)_\C \simeq \g \oplus \g$.
The Cartan involution $\theta$ acts on
$P_{(\g_\R)_\C^\ast} \simeq P_{\g^\ast} \oplus P_{\g^\ast}$
by $(\alpha_1, \alpha_2) \mapsto (\alpha_2, \alpha_1)$
($\alpha_1, \alpha_2 \in P_{\g^\ast}$). Therefore
$\dim (P_{(\g_\R)_\C^\ast}^k)^{-\theta} = \dim P_{\g^\ast}^k$
for every $k \in \N$, and Proposition~\ref{prop:degree}
in this case follows from Table~\ref{table:degree}.

Suppose that $\g = \sl(n, \R)$.
Then $(\g_\C, \k_\C) = (\sl(n, \C), \so(n, \C))$.
Let $f_k \in (S^k \g_\C^\ast)^{\g_\C}$ ($2 \leq k \leq n$)
be as in Fact~\ref{fact:f_k}.
The restriction of $f_k$ to $\k_\C$ vanishes if and only if
$k = 3, 5, \dots, 2 \lfloor \frac{n+1}{2} \rfloor -1$.
Since the images of $f_k$ ($2 \leq k \leq n$)
under the Cartan map form a basis of $P_{\g_\C^\ast}$,
we obtain from Fact~\ref{fact:GHV-X7-1}
that the degrees of a basis of $(P_{\g_\C^\ast})^{-\theta}$
are $5, 9, 13, \dots, 4 \lfloor \frac{n+1}{2} \rfloor -3$.

Suppose that $\g = \sl(n, \H)$.
Then $(\g_\C, \k_\C) = (\sl(2n, \C), \sp(n, \C))$.
The restriction of $f_k \in (S^k \g_\C^\ast)^{\g_\C}$
($2 \leq k \leq 2n$) to $\k_\C$ vanishes if and only if
$k = 3, 5, \dots, 2n-1$.
Thus, the degrees of a basis of $(P_{\g_\C^\ast})^{-\theta}$ are
$5, 9, 13, \dots, 4n-3$.

Suppose that $\g = \so(p, q)$, where $p$ and $q$ are odd.
Then $(\g_\C, \k_\C) = (\so(p+q, \C), \so(p,\C) \oplus \so(q,\C))$.
The restriction of $f_{2k} \in (S^{2k} \g_\C^\ast)^{\g_\C}$
($1 \leq k \leq \frac{p+q}{2} -1$) to $\k_\C$
is nonzero for every $k$, and that of
$\tilde{f} \in (S^{(p+q)/2} \g_\C^\ast)^{\g_\C}$ is zero.
Thus, $(P_{\g_\C^\ast})^{-\theta}$
is a 1-dimensional vector space concentrated in degree $p+q-1$.

Suppose that $\g = \e_{6(6)}$.
Then $(\g_\C, \k_\C) = (\e_{6,\C}, \sp(4,\C))$.
Let us fix algebraically independent generators
$g_2, g_5, g_6, g_8, g_9, g_{12}$ ($\deg g_k = k$) of the algebra
$(S \g_\C^\ast)^{\g_\C}$ (cf.\ Table~\ref{table:degree-sym}).
Notice from Table~\ref{table:degree-sym} that
the degrees of generators of $(S \k_\C^\ast)^{\k_\C}$
are all even.
Therefore, $(S^k \k_\C^\ast)^{\k_\C} = 0$ if $k$ is odd,
and in particular, $g_5|_{\k_\C} = g_9|_{\k_\C} = 0$.
By Fact~\ref{fact:GHV-X7-1}, the images of $g_5$ and $g_9$
under the Cartan map are nonzero elements of
$(P_{\g_\C^\ast})^{-\theta}$. Their degrees are 9 and 17,
respectively.
Since
$\dim (P_{\g_\C^\ast})^{-\theta} = \rank \g_\C - \rank \k_\C = 2$
(Fact~\ref{fact:GHV-X7-2}), they form a basis of
$(P_{\g_\C^\ast})^{-\theta}$.

Finally, suppose that $\g = \e_{6(-26)}$.
Then $(\g_\C, \k_\C) = (\e_{6,\C}, \f_{4,\C})$.
Again, $(S^k \k_\C^\ast)^{\k_\C} = 0$ if $k$ is odd, and
$\rank \g_\C - \rank \k_\C = 2$. The degrees of
$(P_{\g_\C^\ast})^{-\theta}$ are hence 9 and 17
by the same argument as the case of $\g = \e_{6(6)}$.
\end{proof}

\section{Some results on the surjectivity of $\rest : (P_{\g^\ast})^{-\theta} \to (P_{\h^\ast})^{-\theta}$}

Our goal is to classify the semisimple symmetric pairs $(\g, \h)$
such that $i : H^\bl(\g, \h; \R) \to H^\bl(\g, \k_H; \R)$
is injective.
In view of Fact~\ref{fact:CO}, it suffices to
classify $(\g, \h)$ such that restriction map
$\rest: (P_{\g^\ast})^{-\theta} \to (P_{\h^\ast})^{-\theta}$
is surjective.
In this section, we give some easy-to-check necessary conditions
and sufficient conditions for the surjectivity.

Although our main interest in this paper is the case of
semisimple symmetric spaces, many results in this section are
valid, more generally, for homogeneous spaces of reductive type.
Thus, we state some direct consequence of
these results in a general setting (Remark~\ref{rem:rank},
Corollaries~\ref{cor:d_k}, \ref{cor:complex} and
\ref{cor:real_form}),
while they are not used in this paper.

\subsection{Comparison of the total dimensions}

As we saw in \cite{Mor17+},
we obtain the following necessary condition for
the surjectivity of the restriction map
$\rest : (P_{\g^\ast})^{-\theta} \to (P_{\h^\ast})^{-\theta}$
from Fact~\ref{fact:GHV-X7-2}:
\begin{prop}\label{prop:rank}
Let $(\g, \h)$ be a reductive pair with Cartan involution $\theta$.
If $\rank \g - \rank \k < \rank \h - \rank \k_H$,
the restriction map
$\rest : (P_{\g^\ast})^{-\theta} \to (P_{\h^\ast})^{-\theta}$
is not surjective.
\end{prop}
\begin{proof}
By Fact~\ref{fact:GHV-X7-2},
\[
\dim (P_{\g^\ast})^{-\theta} = \rank \g - \rank \k
< \rank \h - \rank \k_H = \dim (P_{\h^\ast})^{-\theta}. \qedhere
\]
\end{proof}
\begin{rem}\label{rem:rank}
Kobayashi's rank conjecture~\cite[Conj.~6.4]{Kob89Geom}
follows immediately from Facts~\ref{fact:previous}, \ref{fact:CO}
and Proposition~\ref{prop:rank} (see \cite{Mor17+}).
\end{rem}

\subsection{Comparison of the dimensions of low degree parts}

Let $\g$ be a reductive Lie algebra and write $\g$ as
\[
\g \simeq \R^{n^+} \oplus \sqrt{-1}\R^{n^-} \oplus \bigoplus_{\substack{\text{$\l$: simple} \\ \text{Lie algebra}}} m_{\l} \cdot \l.
\]
We then put
\begin{align*}
d_1(\g) &= n^+, \\
d_2(\g) &= \sum_{\substack{\text{$\l$: complex simple} \\ \text{Lie algebra}}} m_{\l}, \\
d_3(\g) &= \sum_{k \geq 3} m_{\sl(k,\R)} + \sum_{k \geq 3} m_{\sl(k,\C)} + \sum_{k \geq 2} m_{\sl(k,\H)}, \\
d_4(\g) &= m_{\so(7,1)} + m_{\so(5,3)} + 2m_{\so(8,\C)} \\
&\qquad + \sum_{k \geq 4} m_{\sl(k,\C)} + \sum_{k \geq 7, \, k \neq 8} m_{\so(k,\C)} + \sum_{k \geq 2} m_{\sp(k,\C)}.
\end{align*}
\begin{rem}
When we compute $d_k(\g)$,
we must be careful in the following accidental isomorphisms:
\begin{itemize}
\item $d_1(\g)$:
$\so(2,\C) \simeq \C = \R \oplus \sqrt{-1}\R$,
$\so(1,1) \simeq \R$.
\item $d_2(\g)$:
$\so(4,\C) \simeq \sl(2,\C) \oplus \sl(2,\C)$,
$\so(3,1) \simeq \sl(2,\C)$,
\item $d_3(\g)$:
$\so(6,\C) \simeq \sl(4,\C)$,
$\so(3,3) \simeq \sl(4,\R)$,
$\so(5,1) \simeq \sl(2,\H)$.
\item $d_4(\g)$:
$\so(5,\C) \simeq \sp(2,\C)$,
$\so(6,\C) \simeq \sl(4,\C)$.
\end{itemize}
\end{rem}
\begin{prop}\label{prop:d_k}
Let $(\g, \h)$ be a reductive pair with Cartan involution $\theta$.
If $d_k(\g) < d_k(\h)$ for some $1 \leq k \leq 4$,
the restriction map
$\rest : (P_{\g^\ast})^{-\theta} \to (P_{\h^\ast})^{-\theta}$
is not surjective.
\end{prop}
\begin{proof}
It is enough to check that
$\rest :
(P_{\g^\ast}^{2k-1})^{-\theta} \to (P_{\h^\ast}^{2k-1})^{-\theta}$
is not injective.
By Proposition~\ref{prop:degree}, we have
\[ \dim (P_{\g^\ast}^{2k-1})^{-\theta} =  d_k(\g) <
d_k(\h) = \dim (P_{\h^\ast}^{2k-1})^{-\theta}. \qedhere \]
\end{proof}

\begin{cor}\label{cor:d_k}
Let $G/H$ be a homogeneous space of reductive type.
If $d_k(\g) < d_k(\h)$ for some $1 \leq k \leq 4$, then
there does not exists a compact manifold locally modelled on $G/H$.
\end{cor}
\begin{proof}
Combine Facts~\ref{fact:previous}, \ref{fact:CO}
and Proposition~\ref{prop:d_k}.
\end{proof}

\subsection{The $c$-dual of a semisimple symmetric pair}

Let $(\g, \h)$ be a semisimple symmetric pair defined by an
involution $\sigma$ (i.e.\ $\h = \g^\sigma$).
Put $\q = \g^{-\sigma}$ and
\[
\g^c = \h \oplus \sqrt{-1} \q \subset \g_\C,
\]
where $\g_\C$ is the complexification of $\g$.
Then $(\g^c, \h)$ becomes a semisimple symmetric pair.
It is called the $c$-dual of the semisimple symmetric pair
$(\g, \h)$.
If $\theta$ is a Cartan involution of $\g$
that commutes with $\sigma$,
one can define a Cartan involution $\theta^c$ of $\g^c$ by
\[ \theta^c|_{\h} = \theta, \qquad
\theta^c|_{\sqrt{-1} \q} = -\theta. \]

Note that $\g^{cc} = \g$, and
the $c$-dual of an irreducible symmetric pair is again irreducible.

\begin{prop}\label{prop:c-dual}
Let $(\g, \h)$ be a semisimple symmetric pair and
$(\g^c, \h)$ its $c$-dual.
Then, the restriction map
$\rest: (P_{\g^\ast})^{-\theta} \to (P_{\h^\ast})^{-\theta}$
is surjective if and only if
$\rest:
(P_{(\g^c)^\ast})^{-\theta^c} \to (P_{\h^\ast})^{-\theta^c}$
is surjective.
\end{prop}
\begin{proof}
We use the same symbols $\theta$ and $\theta^c$
for their complex linear extensions to $\g_\C$.
The restriction map
$\rest: (P_{\g^\ast})^{-\theta} \to (P_{\h^\ast})^{-\theta}$
is surjective if and only if
\[ \im (\rest: P_{\g^\ast} \to P_{\h^\ast}) \subset
(P_{\h^\ast})^{-\theta}, \]
or equivalently,
\[ \im (\rest: P_{\g_\C^\ast} \to P_{\h_\C^\ast}) \subset
(P_{\h^\ast_\C})^{-\theta}. \]
Similarly,
$\rest: (P_{(\g^c)^\ast})^{-\theta^c} \to
(P_{\h^\ast})^{-\theta^c}$
is surjective if and only if
\[ \im (\rest: P_{\g_\C^\ast} \to P_{\h_\C^\ast}) \subset
(P_{\h^\ast_\C})^{-\theta^c}. \]
Since $\theta = \theta^c$ on $\h_\C$,
these two conditions are equivalent.
\end{proof}

\subsection{Complex reductive pairs}

A reductive pair $(\g, \h)$ is called complex if $\g$
is a complex Lie algebra and $\h$ is a complex subalgebra of $\g$.
We remark that the Cartan involution $\theta$ of $\g$ is
complex antilinear in this case.
We use the notation $\g_\R$ (resp.\ $\h_\R$) when we regard $\g$
(resp.\ $\h$) as a real Lie algebra.
\begin{prop}\label{prop:complex}
Let $(\g, \h)$ be a complex reductive pair with Cartan involution
$\theta$. Then, the following three conditions are equiavlent:
\begin{enumerate}
\item[\textup{(1)}] The (real) linear map
$\rest: (P_{{\g_\R}^\ast})^{-\theta} \to
(P_{{\h_\R}^\ast})^{-\theta}$
is surjective.
\item[\textup{(2)}] The (complex) linear map
$\rest: P_{\g^\ast} \to P_{\h^\ast}$ is surjective.
\item[\textup{(3)}] The (complex) graded algebra homomorphism
$\rest: (S \g^\ast)^\g \to (S \h^\ast)^\h$ is surjective.
\end{enumerate}
\end{prop}
\begin{proof}
(1) $\Leftrightarrow$ (2).
As we saw in the proof of Proposition~\ref{prop:degree},
the Cartan involution $\theta$ acts on
$P_{{\g_\R}^\ast} \ox \C \simeq P_{\g^\ast} \oplus P_{\g^\ast}$
by $(\alpha_1, \alpha_2) \mapsto (\alpha_2, \alpha_1)$
($\alpha_1, \alpha_2 \in P_{\g^\ast}$).
Thus
$(P_{{\g_\R}^\ast})^{-\theta} \ox \C \simeq
\{ (\alpha, -\alpha) : \alpha \in P_{\g^\ast} \}$,
and similarly
$(P_{{\h_\R}^\ast})^{-\theta} \ox \C \simeq
\{ (\beta, -\beta) : \beta \in P_{\g^\ast} \}$.
By these isomorphisms, the restriction map
$\rest:
(P_{{\g_\R}^\ast})^{-\theta} \to (P_{{\h_\R}^\ast})^{-\theta}$
is rewritten as
$(\alpha, -\alpha) \mapsto (\alpha|_\h, -\alpha|_\h)$
($\alpha \in P_{\g^\ast}$).
This is surjective if and only if so is
$\rest: P_{\g^\ast} \to P_{\h^\ast}$.

(2) $\Leftrightarrow$ (3).
Recall that the restriction map
$\rest: P_{{\g}^\ast} \to P_{{\h}^\ast}$
is canonically identified with
\[
\overline{\rest}:
(S^+ \g^\ast)^\g/((S^+ \g^\ast)^\g \cdot (S^+ \g^\ast)^\g) \to
(S^+ \h^\ast)^\h/((S^+ \h^\ast)^\h \cdot (S^+ \h^\ast)^\h)
\]
via the Cartan maps.
By Fact~\ref{fact:CST},
the above linear map is surjective if and only if
$\rest: (S \g^\ast)^\g \to (S \h^\ast)^\h$ is surjective.
\end{proof}

\begin{cor}\label{cor:complex}
Let $G/H$ be a complex homogeneous space of reductive type
(i.e.\ a homogeneous space of reductive type
such that the corresponding reductive pair $(\g, \h)$ is complex).
If the restriction map $\rest : P_{\g^\ast} \to P_{\h^\ast}$
is not surjective or, equivalently,
$\rest : (S \g^\ast)^\g \to (S \h^\ast)^\h$ is not surjective,
then there does not exists a compact manifold locally modelled on
$G/H$.
\end{cor}
\begin{proof}
Combine Facts~\ref{fact:previous}, \ref{fact:CO}
and Proposition~\ref{prop:complex}.
\end{proof}

\subsection{Complexifications}

\begin{prop}\label{prop:real_form}
Let $(\g_0, \h_0)$ be a reductive pair with Cartan involution
$\theta_0$. Let $\g$ and $\h$ be the complexifications of
$\g_0$ and $\h_0$, respectively.
We denote by $\theta$ the Cartan involution of $\g$ such that
$\theta|_\g = \theta_0$.
Then, the restriction map
$\rest: (P_{\g_\R^\ast})^{-\theta} \to (P_{\h_\R^\ast})^{-\theta}$
is surjective only if so is
$\rest:
(P_{\g_0^\ast})^{-\theta_0} \to (P_{\h_0^\ast})^{-\theta_0}$.
\end{prop}
\begin{proof}
We write $\theta_\C$ for the complex linear extension of
$\theta_0$ to $\g$.
We note that $\theta_\C$ is a complex linear involution on $\g$,
whereas $\theta$ is complex antilinear. If the linear map
$\rest: (P_{\g_\R^\ast})^{-\theta} \to (P_{\h_\R^\ast})^{-\theta}$
is surjective, so is $\rest : P_{\g^\ast} \to P_{\h^\ast}$
by Proposition~\ref{prop:complex} (1)~$\Rightarrow$~(2).
In particular, it is surjective on the $(-1)$-eigenspaces for
$\theta_\C$, namely,
\[
\rest : (P_{\g_0^\ast})^{-\theta_0} \ox \C =
(P_{\g^\ast})^{-\theta_\C} \to (P_{\h^\ast})^{-\theta_\C} =
(P_{\h_0^\ast})^{-\theta_0} \ox \C
\]
is surjective.
This is equivalent to saying that
$\rest :
(P_{\g_0^\ast})^{-\theta_0} \to (P_{\h_0^\ast})^{-\theta_0}$
is surjective.
\end{proof}
\begin{cor}\label{cor:real_form}
Let $G_0/H_0$ be a homogeneous space of reductive type
with Cartan involution $\theta_0$.
Let $G/H$ be a complex homogeneous space of reductive type
such that $G$ and $H$ are the complexifications of
$G_0$ and $H_0$, respectively.
If the homomorphism
$i : H^\bl(\g_0, \h_0; \R) \to H^\bl(\g_0, \k_{H_0}; \R)$
is not injective or, equivalently,
$\rest : (P_{\g^\ast})^{-\theta} \to (P_{\h^\ast})^{-\theta}$
is not surjective, then there does not exist
a compact manifold locally modelled on $G/H$.
\end{cor}
\begin{proof}
Combine Facts~\ref{fact:previous}, \ref{fact:CO}
and Proposition~\ref{prop:real_form}.
\end{proof}

\section{Proof of Theorem~\ref{thm:classify}}\label{sect:proof}

Notice that a direct sum of two semisimple symmetric spaces
$(\g_1, \h_1)$ and $(\g_2, \h_2)$ satisfies the condition (A)
(resp.\ (B)) if and only if both $(\g_1, \h_1)$ and $(\g_2, \h_2)$
satisfy (A) (resp.\ (B)).
Recall from Fact~\ref{fact:CO}
that the injectivity of the homomorphism
$i : H^\bl(\g, \h; \R) \to H^\bl(\g, \k_H; \R)$
is equivalent to the surjectivity of the linear map
$\rest: (P_{\g^\ast})^{-\theta} \to (P_{\h^\ast})^{-\theta}$.
Therefore, it is sufficient to prove the following two claims:
\begin{itemize}
\item If an irreducible symmetric pair $(\g, \h)$
is listed in (B-1)--(B-5), the restriction map
$\rest: (P_{\g^\ast})^{-\theta} \to (P_{\h^\ast})^{-\theta}$
is surjective.
\item If an irreducible symmetric pair $(\g, \h)$
is listed in Table~\ref{table:classify}, the restriction map
$\rest: (P_{\g^\ast})^{-\theta} \to (P_{\h^\ast})^{-\theta}$
is not surjective.
\end{itemize}

If $(\g, \h)$ satisfies (B-1) or (B-2), the surjectivity of
$\rest: (P_{\g^\ast})^{-\theta} \to (P_{\h^\ast})^{-\theta}$
is obvious.

If $(\g, \h)$ satisfies (B-3), its $c$-dual $(\g^c, \h)$
satisfies (B-2), and therefore the restriction map
$\rest: (P_{\g^\ast})^{-\theta} \to (P_{\h^\ast})^{-\theta}$
is surjective by Proposition~\ref{prop:c-dual}.

If $(\g, \h)$ satisfies (B-4), we have
$\dim (P_\h)^{-\theta} = \rank \h - \rank \k_H = 0$
by Fact~\ref{fact:GHV-X7-2}. Hence,
$\rest: (P_{\g^\ast})^{-\theta} \to (P_{\h^\ast})^{-\theta}$
is trivially surjective.

Suppose that $(\g, \h)$ satisfies (B-5).
By Proposition~\ref{prop:complex}~(3)~$\Rightarrow$~(1),
it suffices to see that
$\rest : (S \g^\ast)^\g \to (S \h^\ast)^\h$ is surjective.
If $(\g, \h) = (\sl(2n+1, \C), \so(2n+1, \C))$,
$(\sl(2n, \C), \sp(n, \C))$ or $(\so(2n, \C), \so(2n-1, \C))$,
the surjectivity easily follows from Fact~\ref{fact:f_k}.
If $(\g, \h) = (\e_{6, \C}, \f_{4, \C})$,
the surjectivity is proved in \cite[p.\ 322]{Tak62}.

Let $(\g, \h)$ be one of the irreducible symmetric pairs listed in
Table~\ref{table:classify}. If $(\g, \h)$ satisfies
$d_k(\g) < d_k(\h)$ for some $1 \leq k \leq 4$,
the restriction map
$\rest: (P_{\g^\ast})^{-\theta} \to (P_{\h^\ast})^{-\theta}$
is not surjective by Proposition~\ref{prop:d_k}.
The remaining cases are the following:
\begin{center}
\begin{longtable}{lll} \toprule
\multicolumn{1}{c}{$\g$} & \multicolumn{1}{c}{$\h$} & \multicolumn{1}{c}{Conditions} \\ \midrule
$\sl(2n, \C)$ & $\so(2n, \C)$
     & $n \geq 1$ \\ \hline
$\sl(p+q, \R)$ & $\so(p, q)$
     & $p,q \geq 1$, $p,q$: odd \\ \hline
$\su(p, q)$ & $\so(p, q)$
     & $p,q \geq 1$, $p,q$: odd \\ \hline
$\so(2n+1, \C)$ & $\so(2n, \C)$
     & $n \geq 3$ \\ \hline
   & & $p,q \geq 1$, $p,q$: odd, \\
$\so(p+r,q+s)$ & $\so(p,q) \oplus \so(r,s)$
     & $r,s \geq 0$, $(r,s) \neq (0,0)$, \\
   & & $(p,q,r,s) \neq (1,1,1,1)$ \\
\bottomrule
\caption{$(\g, \h)$ listed in Table~\ref{table:classify} to which Proposition~\ref{prop:d_k} is not applicable}
\label{table:remaining}
\end{longtable}
\end{center}

\vspace{-1.2cm}

The pair $(\sl(p+q, \R), \so(p,q))$ is the $c$-dual of
$(\su(p,q), \so(p,q))$.
The pairs $(\sl(2n, \C), \so(2n, \C))$ and
$(\so(2n+1, \C), \so(2n, \C))$ are the complexifications of
$(\su(2n-1, 1), \so(2n-1, 1))$ and $(\so(2n, 1), \so(2n-1, 1))$,
respectively.
Hence, by Propositions~\ref{prop:c-dual} and \ref{prop:real_form},
it is sufficient to verify that the restriction map
$\rest: (P_{\g^\ast})^{-\theta} \to (P_{\h^\ast})^{-\theta}$
is not surjective when $(\g, \h)$ is $(\su(p, q), \so(p, q))$
($p,q$: odd) or
$(\g, \h) = (\so(p+r,q+s), \so(p,q) \oplus \so(r,s))$
($p,q$: odd, $(r,s) \neq (0,0)$).
If $(\g, \h) = (\su(p, q), \so(p, q))$ ($p,q$: odd),
the nonsurjectivity follows from Proposition~\ref{prop:rank}.
If $(\g, \h) = (\so(p+r,q+s), \so(p,q) \oplus \so(r,s))$
($p,q$: odd, $(r,s) \neq (0,0)$),
we have
$\dim (P^{p+q-1}_{\g^\ast})^{-\theta} <
\dim (P^{p+q-1}_{\h^\ast})^{-\theta}$,
and the map
$\rest: (P_{\g^\ast})^{-\theta} \to (P_{\h^\ast})^{-\theta}$
cannot be surjective. \qed

\appendix

\section{The nonlinear case}

Let $G/H$ be a homogeneous space and $\pi : \tilde{G} \to G$
a covering map.
Let $\tilde{H} = \pi^{-1}(H)$.
Then, $\tilde{G}/\tilde{H}$ is diffeomorphic to $G/H$ via $\pi$.
In this appendix, we discuss the difference between
the existence problem of a compact Clifford--Klein form of
$\tilde{G}/\tilde{H}$ and of $G/H$
(or, more generally, the existence problem of a compact manifold
locally modelled on $\tilde{G}/\tilde{H}$ and on $G/H$).

\begin{rem}
Recall that we have assumed the cocycle condition for the
transition functions (\cite[\S2]{Mor17PRIMS}) in the definition of
a manifold locally modelled on a homogeneous space.
Without this assumption,
all the discussion in this appendix is false.
Note that Clifford--Klein forms always satisfy
the cocycle condition.
\end{rem}

Every manifold locally modelled on $\tilde{G}/\tilde{H}$
can be naturally regarded as a manifold locally modelled on $G/H$.
In particular, if there does not exist a compact manifold locally
modelled on $G/H$, neither does on $\tilde{G}/\tilde{H}$.
On the other hand, it is not necessarily possible to give a
structure of a manifold locally modelled on $\tilde{G}/\tilde{H}$
to a manifold locally modelled on $G/H$: in general,
one cannot lift transition functions from $G$ to $\tilde{G}$
so that the cocycle condition is still satisfied.
Thus, even if there exists a compact manifold locally modelled on
$G/H$, it does not imply the existence of a compact manifold
locally modelled on $\tilde{G}/\tilde{H}$.

The situation is similar for Clifford--Klein forms.
If $\tilde{\Gamma} \bs \tilde{G}/\tilde{H}$
is a Clifford--Klein form (i.e.\ if $\tilde{\Gamma}$
is a discrete subgroup of $\tilde{G}$
acting properly and freely on $G/H$),
$\pi(\tilde{\Gamma}) \bs G/H$ is also a Clifford--Klein form.
However, even if $\Gamma \bs G/H$ is a Clifford--Klein form,
$\pi^{-1}(\Gamma) \bs \tilde{G}/\tilde{H}$
is not a Clifford--Klein form except when $\tilde{G} = G$:
indeed, the $\pi^{-1}(\Gamma)$-action on $\tilde{G}/\tilde{H}$
is not free, and if $\pi$ is an infinite covering,
it is not even proper.
Note that, when $\Gamma$ is finitely generated (this is always
satisfied if the Clifford--Klein form $\Gamma \bs G/H$ is compact)
and $\tilde{G}$ is linear, one can circumvent the freeness issue
by Selberg's lemma~\cite[Lem.~8]{Sel60}. Thus, if $\tilde{G}$ is
a finite linear covering of $G$,
the existence problem of compact Clifford--Klein forms of
$\tilde{G}/\tilde{H}$ is equivalent to that of $G/H$.

Now, let us state our result. Let $G/K$ be
an irreducible Hermitian symmetric space of noncompact type
(i.e.\ $G$ a connected linear simple Lie group whose fundamental
group is infinite and $K$ its maximal compact subgroup).
By A.\ Borel's theorem \cite{Bor63},
$G/K$ admits a compact Clifford--Klein form.
In contrast, we have the following:

\begin{prop}\label{prop:nonlinear}
Let $G/K$ be an irreducible Hermitian symmetric space
of noncompact type. Let $\pi : \tilde{G} \to G$
be the universal covering map of $G$ and $\tilde{K} = \pi^{-1}(K)$.
Then, there does not exist a compact manifold locally modelled on
$\tilde{G}/\tilde{K}$. In particular, $\tilde{G}/\tilde{K}$
does not admit a compact Clifford--Klein form.
\end{prop}

\begin{proof}
We shall apply the following generalization of
Fact~\ref{fact:previous}:
\begin{fact}[{\cite[Th.~1.2\ (2) and Prop.~5.1]{Mor17PRIMS}}]
Let $G$ be a (possibly nonlinear) Lie group such that
the adjoint action of $\g$ on itself is trace-free.
Let $H$ be a closed subgroup of $G$ with finitely many
connected components such that $\h$ is reductive in $\g$
(i.e.\ that $\g$ is completely reducible as an $\h$-module).
Let $K_H$ be a maximal compact subgroup of $H$.
If the homomorphism $i : H^\bl(\g, \h; \R) \to H^\bl(\g, \k_H; \R)$
is not injective,
there does not exist a compact manifold locally modelled on $G/H$.
\end{fact}
Notice that $\tilde{K}$ is noncompact since
$\pi$ is an infinite covering.
Let $\tilde{K}_{ss}$ be a connected Lie subgroup of $K$
corresponding to $\k_{ss} = [\k, \k]$.
It follows easily from the Cartan--Malcev--Iwasawa--Mostow theorem
(see e.g.\ \cite[Ch.~VII, Th.~1.2]{Bor98} or
\cite[Ch.~XV, Th.~3.1]{Hoc65})
that $\tilde{K}_{ss}$ is a maximal compact subgroup of $G$.
Thus, it suffices to see that the homomorphism
\[
i: H^\bl(\g, \k; \R) \to H^\bl(\g, \k_{ss}; \R)
\]
is not injective.
Take a nonzero element $X$ of the centre of $\k$.
Let $F$ be the element of $\g^\ast$ corresponding to $X$
via the Killing form of $\g$.
We have $\Stab_\g(F) = \k$.
Then, $\omega = dF$ is an element of
$(\Lambda^2 (\g/\k)^\ast)^{\k}$. Note that $\omega$ corresponds to
the $G$-invariant K\"ahler form on $G/K$ under the isomorphism
$(\Lambda^2 (\g/\k)^\ast)^{\k} \simeq \Omega^2(G/K)^G$.
Since $\omega$ is a nondegenerate symmetric form on $\g/\k$,
we have $\omega^{N/2} \neq 0$, where $N = \dim(G/K)$.
Since $H^N(\g, \k; \R) \simeq \R$,
it follows that $[\omega^{N/2}]_{\g, \k} \neq 0$
in $H^N(\g, \k; \R)$.
In particular, $[\omega]_{\g, \k} \neq 0$ in $H^2(\g, \k; \R)$.
On the other hand, since $F \in ((\g/\k_{ss})^\ast)^{\k_{ss}}$,
we have
$[\omega]_{\g, \k_{ss}} = [dF]_{\g, \k_{ss}} = 0$
in $H^2(\g, \k_{ss}; \R)$.
Therefore, the above homomorphism $i$ is not injective.
\end{proof}

\begin{rem}
The homogeneous space $\tilde{G}/\tilde{H}$
might be a rather unnatural object from a geometric viewpoint:
the $\tilde{G}$-action on $\tilde{G}/\tilde{H}$
is not effective unless $\tilde{G} = G$.
However, it sometimes appears in the study of the Lie group
$\tilde{G}$. For instance, in the erratum of \cite{Ati-Sch77},
Atiyah--Schmid considered a homogeneous space $\tilde{G}/\tilde{K}$
in order to construct the discrete series representations of
$\tilde{G}$, where $\tilde{G}$ is a nonlinear finite covering of
a linear semisimple Lie group $G$ and
$K$ is a maximal compact subgroup of $G$.
In fact, they needed to study the nonlinear case separately since
it is not known if $\tilde{G}/\tilde{K}$
admits a compact Clifford--Klein form, unlike the linear case
to which A. Borel's theorem~\cite{Bor63} applies.
Note that Proposition~\ref{prop:nonlinear}
is not applicable to the case which Atiyah--Schmid studied;
this proposition concerns the case of a nonlinear \emph{infinite}
covering of a linear semisimple Lie group $G$.
\end{rem}

\subsection*{Acknowledgements}

The author would like to express his sincere thanks to
Toshiyuki Kobayashi for his advice and encouragement.
This work was supported by
JSPS KAKENHI Grant Numbers 14J08233 and 17H06784,
the Kyoto University Research Fund for Young Scientists (Start-up)
and the Program for Leading Graduate Schools, MEXT, Japan.

\end{document}